\documentclass[a4paper,11pt,leqno]{article}
\usepackage{amsmath}
\usepackage{amsfonts}
\usepackage{eucal}
\usepackage{amsthm}
\numberwithin{equation}{section}

\newcommand{\bigpare}[1]{\bigl(#1\bigr)}
\newcommand{\biggpare}[1]{\biggl(#1\biggr)}

\newcommand{\biggbrac}[1]{\biggl[#1\biggr]}

\newcommand{\bignorm}[1]{\bigl\| #1 \bigr\|}

\newcommand{\biggnorm}[1]{\biggl\| #1 \biggr\|}

\newcommand{\bigabs}[1]{\bigl| #1 \bigr|}

\newcommand{\biggabs}[1]{\biggl| #1 \biggr|}

\newcommand{\jap}[1]{\langle #1 \rangle}

\def\a{\alpha}
\def\b{\beta}

\def\d{\delta}
\def\e{\varepsilon}
\def\f{\varphi}
\def\g{\psi}

\def\l{\lambda}
\def\m{\mu}
\def\n{\nu}

\def\s{\sigma}
\def\t{\tau}
\def\x{\xi}
\def\y{\eta}
\def\z{\zeta}

\def\re{\mathbb{R}}

\def\ze{\mathbb{Z}}

\def\pa{\partial}

\newcommand{\supp}{\text{{\rm supp}\;}}

\newtheorem{thm}{Theorem}[section]
\newtheorem{lem}[thm]{Lemma}

\theoremstyle{definition}

\newtheorem{ass}{Assumption}

\theoremstyle{remark}
\newtheorem{rem}[thm]{Remark}

\newcommand{\WF}{W\!F}

\title{Wave front set for solutions to 
perturbed harmonic oscillators}

\author{Shikuan MAO\footnote{Graduate School of Mathematical Sciences, 
University of Tokyo, 3-8-1, Komaba, Meguro, Tokyo, Japan 153-8914. 
Email: {\tt  shikuan@ms.u-tokyo.ac.jp}} %
and Shu NAKAMURA\footnote{Graduate School of Mathematical Sciences, 
University of Tokyo, 3-8-1, Komaba, Meguro, Tokyo, Japan 153-8914. 
Partially supported by JSPS Grant Kiban (B) 17340033. 
Email: {\tt  shu@ms.u-tokyo.ac.jp}}}

\begin{document}
\maketitle

\begin{abstract}
We consider Schr\"odinger equations with variable coefficients and
the harmonic potential. We suppose the perturbation is short-range type 
in the sense of \cite{Na1}. 
We characterize the wave front set of the solutions to 
the equation in terms of the classical scattering data 
and the propagator of the unperturbed harmonic oscillator.  
In particular, we give a ``recurrence of singularities'' type theorem 
for the propagation of the period $t=\pi$. 
\end{abstract}


\section{Introduction}
In this paper we consider a Schr\"odinger operator with variable coefficients 
and the harmonic potential: 
\[
H=-\frac12 \sum_{j,k=1}^n \pa_{x_j} a_{jk}(x) \pa_{x_k} +\frac 12 |x|^2 +  V(x)
\]
on $\mathcal{H}= L^2(\re^n)$, $n\geq 1$. We denote the unperturbed harmonic oscillator 
by $H_0$:
\[
H_0=-\frac12 \triangle +\frac12 |x|^2 \qquad \text{on $\mathcal{H}$}, 
\]
and we suppose $H$ is a short-range perturbation of $H_0$ in the 
following sense:

\begin{ass}
$a_{jk}(x)$, $V(x)\in C^\infty(\re^n;\re)$ for $j,k=1,\dots, n$, and 
$(a_{jk}(x))_{j,k}$ is positive symmetric for each $x\in\re^n$. 
Moreover, there exists $\m>1$ such that for any $\a\in\ze_+^n$, 
\begin{align*}
&\bigabs{\pa_x^\a (a_{jk}(x)-\d_{jk})} \leq C_\a\jap{x}^{-\m-|\a|}, \\
& \bigabs{\pa_x^\a V(x)} \leq C_\a\jap{x}^{2-\m-|\a|}
\end{align*}
for $x\in\re^n$ with some $C_\a>0$, where $\jap{x}=\sqrt{1+|x|^2}$. 
\end{ass}

Then it is well-known that $H$ is essentially self-adjoint on $C_0^\infty(\re^n)$, 
and we denote the unique self-adjoint extension by the same symbol $H$. 
We denote the symbols of $H$, $H_0$, the kinetic energy and the free Schr\"odinger 
operator by $p$, $p_0$, $k$ and $k_0$, respectively. Namely, we denote
\begin{align*}
&p(x,\x)=\frac12 \sum_{j,k=1}^n a_{jk}(x)\x_j\x_k +\frac12 |x|^2 +V(x), 
\quad  p_0(x,\x) = \frac12 |\x|^2 +\frac12 |x|^2, \\
& k(x,\x)= \frac12 \sum_{j,k=1}^n a_{jk}(x)\x_j\x_k , 
\quad  k_0(x,\x)=\frac12 |\x|^2. 
\end{align*}
We denote the Hamilton flow generated by a symbol $a(x,\x)$ on $\re^{2n}$ by 
$\exp (tH_a)$ : $\re^{2n}\to\re^{2n}$. We also denote
\[
\pi_1(X)=x, \quad \pi_2(X)=\x \quad \text{for $X=(x,\x)\in\re^{2n}$}.
\]

Let $(x_0,\x_0)\in\re^{2n}$. $(x_0,\x_0)$ is called {\em forward (backward, resp.)
nontrapping}\/ (with respect to $k$) if 
\[
|\pi_1(\exp(tH_k)(x_0,\x_0)| \to \infty
\]
as $t\to+\infty$ ($t\to-\infty$, resp.). If $(x_0,\x_0)$ is forward/backward 
nontrapping, then it is well-known 
\[
(x_\pm,\x_\pm) =\lim_{t\to\pm\infty} \exp(-tH_{k_0})\circ \exp(tH_k)(x_0,\x_0)
\]
exists, and $S_\pm$: $(x_0,\x_0)\mapsto (x_\pm,\x_\pm)$ are locally 
diffeomorphic (see, e.g., Nakamura~\cite{Na1}, Section~2.). 

Now we present our main results of this paper. Let us recall our harmonic oscillator 
$H_0$ has a period $2\pi$, i.e., 
$e^{-i 2\pi H_0}\f =\f$ for $\f\in\mathcal{H}$. Moreover, we have 
\[
e^{\mp i\pi H_0}\f(x)=\f(-x), \quad \f\in\mathcal{H}.
\]
Our first result concern the evolution by $H$ up to time $\pi$. We denote
\[
u(t)= e^{- itH}u_0, \quad u_0\in \mathcal{H}.
\]
We denote the wave front set of a distribution $f$ by $\WF(f)$. 

\begin{thm}
(i) Suppose $(x_0,\x_0)$ is backward nontrapping, and let $0<t_0<\pi$, 
$u_0\in\mathcal{H}$. Then 
\[
(x_0,\x_0)\in \WF(u(t_0)) \quad \Longleftrightarrow \quad (x_-,\x_-)\in \WF(e^{-it_0H_0}u_0).
\]
(ii) Suppose $(x_0,\x_0)$ is forward nontrapping, and let $-\pi<t_0<0$, 
$u_0\in\mathcal{H}$. Then 
\[
(x_0,\x_0)\in \WF(u(t_0)) \quad \Longleftrightarrow \quad (x_+,\x_+)\in \WF(e^{-it_0H_0}u_0).
\]
\end{thm}

\begin{rem}
We note that microlocally $e^{-itH_0}$ is a rotation in the phase space. More precisely, 
for any reasonable symbol $a=a(x,\x)$, 
\[
e^{-itH_0} a^w(x,D_x) e^{itH_0} = a^w\bigpare{\cos(t) x +\sin(t) D_x, -\sin(t) x +\cos(t)D_x}, 
\]
where $a^w(x,D_x)$ denotes the Weyl-quantization of $a$. Hence, in particular, 
$(x_0,\x_0)\in \WF(e^{-itH_0}u_0)$ if and only if there exists a symbol: 
$a\in C_0^\infty(\re^{2n})$ such that $a(x_0,\x_0)\neq 0$ and 
\[
\bignorm{a^w(\cos(t)x-\sin(t)D_x,h(\sin(t)x+\cos(t)D_x))u_0}=O(h^\infty)
\]
as $h\to 0$. 
\end{rem}

At the time $t=\pm\pi$, $u(t)$ behaves differently. We denote the set of forward 
(backward, resp.) nontrapping points by $\mathcal{T}_+$ ($\mathcal{T}_-$, resp.) 
$\subset \re^{2n}\setminus 0 :=\{(x,\x)\in\re^{2n},\x\neq 0\}$. 
$S_\pm$ are diffeomorphism from $(\re^{2n}\setminus0)\setminus\mathcal{T}_\pm$ to 
$\re^{2n}\setminus 0$, and hence $S_\pm^{-1}$ are well-defined from $\re^{2n}\setminus0$ 
to $(\re^{2n}\setminus 0)\setminus\mathcal{T}_\pm$. 
We also denote the antipodal map in $\re^{2n}$ by $\Gamma$, i.e., 
$\Gamma(x,\x)=(-x,-\x)$. 

\begin{thm}
(i) Suppose $(x_0,\x_0)$ is backward nontrapping, and let $u_0\in\mathcal{H}$. Then 
\[
(x_0,\x_0)\in\WF(u(\pi))\quad \Longleftrightarrow \quad 
S_+^{-1}\circ\Gamma\circ S_-(x_0,\x_0)\in \WF(u_0).
\]
(ii) Suppose $(x_0,\x_0)$ is forward nontrapping, and let $u_0\in\mathcal{H}$. Then 
\[
(x_0,\x_0)\in\WF(u(-\pi))\quad \Longleftrightarrow \quad 
S_-^{-1}\circ\Gamma\circ S_+(x_0,\x_0)\in \WF(u_0).
\]
\end{thm}

The microlocal singularities of solutions to Schr\"odinger equations have been 
attracted attention during the past years, especially after the publication of the 
break-through paper by Craig-Kappeler-Strauss~\cite{CKS} in 1996
(for more literature, see references of \cite{CKS} and \cite{Na1}). 
On the other hand, 
the singularities of solutions to the harmonic oscillator type Schr\"odinger 
equations have been studied by several authors, including Zelditch~\cite{Ze}, 
Yajima~\cite{Yj}, Kapitanski-Rodnianski-Yajima~\cite{KRY}, Doi~\cite{Do1} 
and Wunsch~\cite{Wu1}. Most of these works concern the case with 
constant coefficients with potential perturbations. 
In particular, if the metric is flat, i.e., if $a_{jk}(x)=\d_{jk}$, then $S_\pm$ is the identity map, 
and Theorem~1.3 recovers results in \cite{Ze}, \cite{Yj}, \cite{KRY} and \cite{Do1}. 
In fact, if the metric is flat, Theorem~1.1 is also obtained by Doi~\cite{Do1}, and 
this paper is partially motivated by this beautiful work. 
Moreover, he also considered a class of long range type perturbations, 
i.e, when $V(x)=O(|x|)$ as $|x|\to\infty$, and demonstrated that a shift of 
singularities occurs. Microlocal smoothing effect for the Schr\"odinger equations on scattering 
manifolds with harmonic potential is studied by Wunsch~\cite{Wu1}.

In a sense, this paper is an analogue of a work by Nakamura~\cite{Na1} where 
the microlocal singularities of asymptotically flat Schr\"odinger equations is 
studied (see also a closely related work by Hassel-Wunsch~\cite{HW} and Ito-Nakamura \cite{ItNa}). 
The main difference (and the novelty) is the analysis of the classical trajectories 
with high energies. In \cite{Na1}, the standard classical scattering theory is  
sufficient to prove the propagation of singularities (with a scaling argument). 
However, in the presence of the harmonic potential, the high energy asymptotics 
of the classical trajectories is completely different from the long-time asymptotics. 
Thus we need to obtain precise high energy asymptotics of the trajectories using 
the time evolution of the harmonic oscillator, and it is carried out in Section~2. 
This situation is somewhat similar to the case of such analysis 
for Schr\"odinger equations with long-range perturbations \cite{Na2}, 
but the asymptotics itself is naturally completely different. 
We also note that our results have much in common with a paper by Zelditch~\cite{Ze},
at least in spirit, and our results may be considered as generalizations of 
his results to variable coefficients cases. 

The main results are proved in Section~3, and the argument is similar to 
\cite{Na1}. However, the scaling argument is slightly more complicated, and 
we try to give a more transparent formulation (see the last part of Section~2 
and the beginning of Section~3). In the last section, we discuss generalizations 
of our main theorems to the case when the harmonic potential is inhomogeneous. 

\medskip
Throughout this paper, we use the following notations: 
Our Hilbert space is $\mathcal{H}=L^2(\re^n)$, and the space of the bounded 
operators on $\mathcal{H}$ is denoted by $\mathcal{L}(\mathcal{H})$. 
The Fourier transform is denoted by 
\[
\hat u(\x) = \mathcal{F}u(\x)= (2\pi)^{-n/2}\int_{\re^n} e^{-ix\cdot\x} u(x)dx,
\]
and the inverse Fourier transform is denoted by $\check u(x)=\mathcal{F}^*u(x)$. 
For a smooth symbol $a(x,\x)$ on $\re^{2n}$, we denote the Weyl quantization by 
$a^w(x,D_x)$, i.e., for $u\in\mathcal{S}(\re^n)$, a Schwartz function on $\re^{n}$, 
\[
a^w(x,D_x)u(x) =(2\pi)^{-n} \iint e^{i(x-y)\cdot\x} a(\tfrac{x+y}{2},\x) u(y) dy\,d\x.
\]
We use the $S(m,g)$-symbol class notation of H\"ormander~\cite{Ho}.


\section{High energy asymptotics of the classical flow}

In this section, we study the high energy behavior of the classical flow generated by $p(x,\x)$. 
More precisely, we consider the properties of $\exp(tH_p)(x_0,\l\x_0)$ as $\l\to+\infty$. 
Throughout this section, we suppose $(x_0,\x_0)$ is forward nontrapping, and consider 
the case $t>0$. The case $t\leq 0$ can be considered similarly. 

For $\l>0$, we write
\begin{align*}
p^\l(x,\x) &= \frac12 \sum_{j,k=1}^n a_{jk}(x)\x_j\x_k +\frac{|x|^2}{2\l^2} +\frac{1}{\l^2}V(x), \\
p_0^\l(x,\x) &= \frac12|\x|^2+\frac{1}{2\l^2} |x|^2. 
\end{align*}
Then, by direct computations, we learn
\begin{align}
\pi_1(\exp(tH_p)(x,\l\x)) &= \pi_1 (\exp(\l t H_{p^\l})(x,\x)),  \\
\pi_2(\exp(t H_p)(x,\l\x)) &= \l \cdot \pi_2(\exp(\l t H_{p^\l})(x,\x)).
\end{align}
Hence, it suffices to consider $\exp(t H_{p^\l})(x,\x)$ for $0\leq t\leq \l t_0$, 
instead of $\exp(t H_p)(x,\l \x)$ for $0\leq t\leq t_0$. 

We note, for each fixed $t\in\re$, 
\begin{equation}
\lim_{\l\to\infty} \exp(tH_{p^\l})(x,\x) = \exp(tH_k)(x,\x)
\end{equation}
by the continuity of the solutions to ODEs with respect to the coefficients. 
Hence, if $t>0$ is large and then $\l>0$ is taken sufficiently large (after fixing $t$),    
$\pi_1(\exp(tH_{p^\l})(x_0,\x_0))$ is far away from the origin by virtue 
of the nontrapping condition. The next lemma claims that this statement holds 
for $0\leq t\leq \l\d$ with sufficiently small $\d>0$. 

\begin{lem}
There exists $\d>0$ and a small neighborhood $\Omega$ of $(x_0,\x_0)$ such that 
\[
|\pi_1(\exp(tH_{p^\l})(x,\x))|\geq c_1 t-c_2
\quad \text{for $0\leq t\leq \l\d$, $(x,\x)\in \Omega$}
\]
with some $c_1,c_2>0$.
\end{lem}

\begin{proof}
In the following, we denote
\[
\exp(tH_{p^\l})(x,\x) = (y^\l(t;x,\x), \y^\l(t;x,\x)).
\]
By the conservation of the energy: $p^\l(y^\l(t),\y^\l(t))=\text{const.}$, and 
the ellipticity of the principal symbol, we easily see
\[
\frac{1}{\l^2}|y^\l(t;x,\x)|^2 +|\y^\l(t;x,\x)|^2 \leq C, \quad (x,\x)\in\Omega, t\in\re,
\]
where $\Omega$ is a small neighborhood of $(x_0,\x_0)$. 
Hence, in particular, we have
\[
|y^\l(t;x,\x)|\leq C\jap{t}, \quad |\y^\l(t;x,\x)|\leq C
\]
for $t>0$ with some $C>0$. On the other hand, by direct computations, we have 
\begin{align*}
\frac{d^2}{dt^2}|y^\l(t)|^2 &= 2\frac{d^2}{dt^2}\biggpare{y^\l\cdot\frac{d y^\l}{dt}}
=2\frac{d}{dt}\biggpare{\sum_{j,k} a_{jk}(y^\l)y_j^\l \y_k^\l} \\
&= 4p^\l(y^\l,\y^\l)+2 W(y^\l,\y^\l), 
\end{align*}
where 
\begin{align*}
W(y^\l,\y^\l) &= \sum_{j,k,\ell} a_{jk}(y^\l)(a_{k\ell}(y^\l)-\d_{k\ell})\y^\l_j\y^\l_k \\
&\quad+ \sum_{j,k,\ell,m} \frac{\pa a_{jk}}{\pa x_\ell}(y^\l) \, a_{\ell m}(y^\l) 
\y^\l_my^\l_j\y^\l_k \\
&\quad - \sum_{j,k,\ell,m} a_{jk}(y^\l) \frac{\pa a_{\ell m}}{\pa x_k}(y^\l)
 y^\l_j \y^\l_\ell \y^\l_m 
-\frac{1}{\l^2}\sum_{j,k}a_{jk}(y^\l) y^\l_j y^\l_k \\
&\quad -\frac{1}{\l^2} \sum_{j,k} a_{jk}(y^\l) \frac{\pa V}{\pa x_k}(y^\l) y^\l_j 
-\frac{1}{\l^2}|y^\l|^2 - \frac{2}{\l^2} V(y^\l).
\end{align*}
Combining these, we learn 
\[
\frac{d^2}{dt^2} |y^\l(t)|^2\geq 4p^\l(y^\l,\y^\l)-c_4 (\jap{y^\l}^{-\m}
+\l^{-2}\jap{y^\l}^2).
\]
We note $p^\l(x_0,\x_0)=k(x_0,\x_0)+O(\l^{-2})$ and $k(x_0,\x_0)>0$, 
and hence $p^\l(x_0,\x_0)>0$ for large $\l$. 
Since $\l^{-2}\jap{y^\l}^2 =O(\jap{t/\l}^2)$, if $0\leq t\le \d \l$ with sufficiently small
$\d>0$, the last term of the right hand side of the above inequality is small and 
\[
\frac{d^2}{dt^2} |y^\l(t)|^2\geq 3p^\l(y^\l,\y^\l)-c_4 \jap{y^\l}^{-\m}
\]
for the initial condition  $(x,\x)\in \Omega$. 
By the nontrapping condition and (2.3), if $T_0>0$ is sufficiently large and $\l$ is large 
(depending on  $T_0$), then 
\[
c_4 \jap{y^\l(T_0)}^{-\m} \leq p^\l(x,\x) \quad \text{for $(x,\x)\in\Omega$, and}
\quad \frac{d}{dt}|y^\l(T_0)|>0. 
\]
Then by the standard convexity argument, we learn 
\[
|y^\l(t)|^2 \geq |y^\l(T_0)|^2 + p^\l(x,\x)(t-T_0)^2
\quad \text{for $t\in [T_0,\d \l]$}, 
\]
and this implies the assertion. 
\end{proof}

\begin{lem}
Let $\d>0$ and $\Omega$ as in the previous lemma, and let $\s\in (0,\d)$. 
Then 
\[
\lim_{\l\to\infty} \exp(-\s\l H_{p^\l_0})\circ \exp(\s\l H_{p^\l})(x,\x)= S_+(x,\x), 
\]
for $(x,\x)\in\Omega$.
\end{lem}

\begin{proof}
We denote
\[
(z^\l(t;x,\x),\z^\l(t;x,\x))= \exp(-t H_{p_0^\l})\circ \exp(tH_{p^\l})(x,\x),
\]
and we show the convergence of $(z^\l(\s\l),\z^\l(\s\l))$ to $S_+(x,\x)$
for $(x,\x)\in\Omega$. We recall 
\[
\exp(-tH_{p_0^\l})(x,\x) = \bigpare{\cos(\tfrac{t}{\l})x-\l\sin(\tfrac{t}{\l})\x, 
\tfrac{1}{\l}\sin(\tfrac{t}{\l})x +\cos(\tfrac{t}{\l})\x}
\]
since $p_0^\l$ is the scaled harmonic oscillator. Thus 
\[
z^\l(t) =\cos(\tfrac{t}{\l})y^\l(t)-\l\sin(\tfrac{t}{\l})\y^\l(t), \quad
\z^\l = \tfrac{1}{\l}\sin(\tfrac{t}{\l})y^\l (t)+\cos(\tfrac{t}{\l})\y^\l(t).
\]
By direct computations, we have 
\begin{align}
\frac{d}{dt} z^\l_k &= -\frac{1}{\l} \sin(\tfrac{t}{\l}) y^\l_k 
+ \cos(\tfrac{t}{\l})\frac{d y^\l_k}{dt} -\cos(\tfrac{t}{\l})\y^\l_k 
-\l \sin(\tfrac{t}{\l})\frac{d\y^\l_k}{dt}  \\
&=-\frac{1}{\l} \sin(\tfrac{t}{\l})y^\l_k +\cos(\tfrac{t}{\l})\sum_j a_{jk}(y^\l) \y^\l_j 
-\cos(\tfrac{t}{\l}) \y^\l_k +\frac{1}{\l} \sin(\tfrac{t}{\l})y^\l_k \nonumber \\
&\quad + \sin(\tfrac{t}{\l})\biggpare{\frac{\l}{2}\sum_{i,j} \frac{\pa a_{ij}}{\pa x_k}
(y^\l)\y^\l_i\y^\l_j +\frac{1}{\l} \frac{\pa V}{\pa x_k}(y^\l)} \nonumber \\
&= \cos(\tfrac{t}{\l})\sum_j (a_{jk}(y^\l)-\d_{jk})\y^\l_j \nonumber \\
&\quad +\sin(\tfrac{t}{\l}) \biggpare{\frac{\l}{2}\sum_{i,j} \frac{\pa a_{ij}}{\pa x_k}(y^\l)
\y^\l_i\y^\l_j +\frac{1}{\l} \frac{\pa V}{\pa x_k}(y^\l)} \nonumber \\
&= O(\jap{y^\l}^{-\m}) + O(\l \jap{y^\l}^{-\m-1}+\l^{-1}\jap{y^\l}^{1-\m}) \nonumber \\
&= O(\jap{t}^{-\m})\nonumber 
\end{align}
for $0\leq t\leq \d\l$. Similarly, we have 
\begin{align}
\frac{d}{dt} \z^\l_k
&= \frac{1}{\l^2} \cos(\tfrac{t}{\l}) y^\l_k 
+\frac{1}{\l} \sin(\tfrac{t}{\l})\frac{d y^\l_k}{ dt}
-\frac{1}{\l}\sin(\tfrac{t}{\l})\y^\l_k +\cos(\tfrac{t}{\l})\frac{d\y^\l_k}{dt} \\
&= \frac{1}{\l^2}\cos(\tfrac{t}{\l}) y^\l_k +\frac{1}{\l} \sum_j a_{jk}(y^\l) \y^\l_j 
-\frac{1}{\l} \sin(\tfrac{t}{\l})\y^\l_k -\frac{1}{\l^2} \cos(\tfrac{t}{\l}) y^\l_k \nonumber \\
&\quad - \cos(\tfrac{t}{\l})\biggpare{\frac12 \sum_{i,j} \frac{\pa a_{ij}}{\pa x_k}(y^\l)
\y^\l_i\y^\l_j +\frac{1}{\l^2}\frac{\pa V}{\pa x_k}(y^\l) }\nonumber \\
&= \frac{1}{\l} \sin(\tfrac{t}{\l})\sum_j (a_{jk}(y^\l)-\d_{jk}) \y^\l_j \nonumber \\
&\quad -\cos(\tfrac{t}{\l}) \biggpare{\frac12 \sum_{i,j} \frac{\pa a_{ij}}{\pa x_k}(y^\l)
\y^\l_i\y^\l_j +\frac{1}{\l^2}\frac{\pa V}{\pa x_k}(y^\l) } \nonumber \\
&= O(\l^{-1}\jap{y^\l}^{-\m})+O(\jap{y^\l}^{-\m-1} +\l^{-2}\jap{y^\l}^{-(\m-1)} )\nonumber \\
&= O(\jap{t}^{-\m-1})\nonumber 
\end{align}
for $0\leq t\leq \d\l$. Moreover, for each $t\in\re$, we have 
\begin{align*}
\lim_{\l\to\infty}\frac{d}{dt} z^\l_k(t)  
&= \sum_j a_{jk}(\tilde y) \tilde \y_j -\tilde\y_k + \frac{t}{2}\sum_{i,j} 
\frac{\pa a_{ij}}{\pa x_k}(\tilde y)\tilde y_j\tilde\y_j \\
&= \frac{d}{dt}(\tilde y_k -t\tilde \y_k), \\
\lim_{\l\to\infty} \frac{d}{dt}\z^\l_k(t) &= 
-\frac{1}{2} \sum_{i,j} \frac{\pa a_{ij}}{\pa x_k}(\tilde y)\tilde \y_i\tilde\y_j 
= \frac{d}{dt}\tilde \y_k,
\end{align*}
where $(\tilde y(t),\tilde\y(t))= \exp(tH_k)(x,\x)$. By using the dominated 
convergence theorem, we conclude 
\begin{align*}
\lim_{\l\to\infty} z^\l(\s\l) &= x+ \lim_{\l\to\infty} \int_0^{\s\l}
\frac{d z^\l}{dt}dt = x+ \int_0^\infty \frac{d}{dt}(\tilde y-t\tilde\y)dt \\
&= \lim_{t\to+\infty} (\tilde y(t)-t\tilde \y(t)) =\pi_1(S_+(x,\x)), \\
\lim_{\l\to\infty} \z^\l(\s\l) &= \x +\lim_{\l\to\infty} \int_0^{\s\l}
\frac{d\z^\l}{dt} dt  
= \x+\int_0^\infty \frac{d}{dt}\tilde \y(t) dt \\
&= \lim_{\t\to+\infty} \tilde\y(t) = \pi_2(S_+(x,\x)).
\end{align*}
This completes the proof of the lemma. 
\end{proof}

\begin{lem}
Let $0<\s<\pi$, and let $\Omega$ be a small neighborhood of $(x_0,\x_0)$ as in 
the previous lemmas. Then 
\[
\lim_{\l\to\infty} \exp(-\s\l H_{p_0^\l})\circ \exp(\s\l H_{p^\l})(x,\x) =S_+(x,\x)
\]
for $(x,\x)\in \Omega$. 
\end{lem}

\begin{proof}
It suffices to consider the case $\d<\s<\pi$, and we fix such $\s$. 
Let $\e>0$, and we show that if 
\[
\max(|z^\l(\s\l)-x_+|, |\z^\l(\s\l)-\x_+|) >\e
\]
then $\l$ is bounded from above, where $(x_+,\x_+)=S_+(x,\x)$, 
and $(z^\l,\z^\l)$ is as in the proof of the previous lemma. 
Our claim then follows from this assertion by contradiction. 

We first note
\[
y^\l(t) =\cos(\tfrac{t}{\l}) z^\l(t) +\l \sin(\tfrac{t}{\l}) \z^\l(t).
\]
For the moment, we suppose $|z^\l(t)-x_+|$, $|\z^\l(t)-\x_+|\leq\e$. Then we have 
\[
|y^\l(t)| \geq \l \sin(\tfrac{t}{\l}) (|\x_+|-\e)-(|x_+|+\e)\geq \d_1\l
\]
with some $\d_1>0$ provided $\d\l\leq t\leq \s\l$, $\e<|\x_+|/2$ and 
$\l\geq \l_0$, where $\d_1$ and $\l_0$ depend only on $|x_+|$ and $|\x_+|$. 
Then, by using formulas (2.4) and (2.5), we learn 
\[
\biggabs{\frac{d}{dt}z^\l(t)}\leq C\l^{-\m}, \quad 
\biggabs{\frac{d}{dt}\z^\l(t)} \leq C \l^{-\m-1}.
\]

Now we choose $\l$ sufficiently large that 
\[
\max(|z^\l(\d\l)-x_+|, |\z^\l(\d\l)-\x_+|)\leq \frac{\e}{2}, 
\]
and suppose 
\[
\max(|z^\l(\s\l)-x_+|, |\z^\l(\s\l)-\x_+|)\geq \e. 
\]
Then there exists $t_0\in (\d\l, \s\l)$ such that 
\[
\max(|z^\l(t_0)-x_+|, |\z^\l(t_0)-\x_+|)=\e
\]
and 
\[
\max(|z^\l(t)-x_+|, |\z^\l(t)-\x_+|)\leq \e\quad \text{for $\d\l\leq t\leq t_0.$}.
\]
By the above observation, we learn 
\begin{align*}
|z^\l(t_0)-x_+| &= \biggabs{z^\l(\d\l)-x_+ +\int_{\d\l}^{t_0} \frac{d z^\l}{dt}(t) dt } \\
&\leq |z^\l(\d\l)-x_+| +C (t_0-\d\l)\l^{-\m} \\
&\leq |z^\l(\d\l)-x_+| +C (\s-\d)\l^{-(\m-1)}, \\
|\z^\l(t_0)-\x_+| &= \biggabs{\z^\l(\d\l)-\x_+ +\int_{\d\l}^{t_0} \frac{d \z^\l}{dt}(t) dt } \\
&\leq |\z^\l(\d\l)-\x_+| +C (\s-\d)\l^{-\m}.
\end{align*}
Thus we have 
\[
\e=\max(|z^\l(t_0)-x_+|, |\z^\l(t_0)-\x_+|) \leq \frac{\e}{2} +C(\s-\d)\l^{-(\m-1)},
\]
and hence $\l \leq (2C(\s-\d)/\e)^{1/(\m-1)}$. This proves the assertion. 
\end{proof}

The next theorem follows immediately from Lemma~2.3. 

\begin{thm}
(i) Suppose $(x_0,\x_0)$ is forward nontrapping, and let $0<\s<\pi$. Then there exists
a neighborhood $\Omega$ of $(x_0,\x_0)$ such that 
\begin{align*}
&\lim_{\l\to\infty} \pi_1\bigpare{\exp(-\s H_{p_0})\circ \exp(\s H_p)(x,\l\x)} 
=\pi_1 (S_+(x,\x)), \\
& \lim_{\l\to\infty} \l^{-1}  \pi_2\bigpare{\exp(-\s H_{p_0})\circ \exp(\s H_p)(x,\l\x)} 
=\pi_2 (S_+(x,\x))
\end{align*}
for $(x,\x)\in\Omega$, and the convergence is uniform in $\Omega$. \newline
(ii) Suppose $(x_0,\x_0)$ be backward nontrapping, and let $-\pi<\s<0$. 
Then there exists a neighborhood $\Omega$ of $(x_0,\x_0)$ such that 
\begin{align*}
&\lim_{\l\to\infty} \pi_1\bigpare{\exp(-\s H_{p_0})\circ \exp(\s H_p)(x,\l\x)} 
=\pi_1 (S_-(x,\x)), \\
& \lim_{\l\to\infty} \l^{-1}  \pi_2\bigpare{\exp(-\s H_{p_0})\circ \exp(\s H_p)(x,\l\x)} 
=\pi_2 (S_-(x,\x))
\end{align*}
for $(x,\x)\in\Omega$, and the convergence is uniform in $\Omega$. 
\end{thm}

We introduce several notations as a preparation for the proof of our main results. 
We set
\begin{align*}
\ell(t;x,\x) &= (p\circ \exp(tH_{p_0}))(x,\x)  -p_0(x,\x) \\
&= \sum_{j,k=1}^n \bigpare{a_{jk}(\cos(t)x+\sin(t)\x)-\d_{jk}}(-\sin(t)x_j+\cos(t)\x_j)\times \\
&\qquad \qquad \times (-\sin(t)x_k+\cos(t)\x_k) +V(\cos(t)x+\sin(t)\x).
\end{align*}
Then it is easy to show that $\ell(t;x,\x)$ generates the {\em scattering time evolution}\/:
\[
S_t =\exp(-tH_{p_0})\circ \exp(tH_{p}).
\]
Similarly, 
\[
\ell^\l(t;x,\x) = (p^\l\circ \exp(tH_{p_0^\l}))(x,\x)-p_0^\l(x,\x)
\]
generates the time evolution:
\[
S^\l_t =\exp(-tH_{p^\l_0})\circ \exp(tH_{p^\l}).
\]

We denote the scaling with respect to $\x$ by $\mathcal{J}_\l$, i.e., 
\[
\mathcal{J}_\l(x,\x)=(x,\l\x)\quad \text{for $(x,\x)\in\re^{2n}$}.
\]
Then by (2.1) and (2.2), we have 
\[
\exp(tH_p)\circ \mathcal{J}_\l = \mathcal{J}_\l \circ \exp(\l t H_{p^\l}), 
\quad 
\exp(tH_{p_0})\circ \mathcal{J}_\l = \mathcal{J}_\l \circ \exp(\l t H_{p_0^\l}), 
\]
and hence we also have 
\begin{equation}
S_t \circ \mathcal{J}_\l = \mathcal{J}_\l\circ S_{\l t}^\l.
\end{equation}


\section{Proof of main results}

In this section, we mainly concern the case $(x_0,\x_0)$ is forward nontrapping, and 
prove the part (ii) of Theorem~1.1. 

We first consider the property of $e^{itH_0}e^{-itH}$ for $t\leq 0$. 
Let $v_0\in C_0^\infty(\re^n)$, and we consider 
\[
v(t)=e^{itH_0} e^{-itH} v_0. 
\]
Then it is easy to observe 
\begin{align*}
\frac{d}{dt} v(t) &= -i e^{itH_0} (H-H_0) e^{-itH} v_0 \\
&= -i \bigpare{ e^{itH_0} H e^{-itH_0} -H_0} v(t) 
= -i L(t) v(t), 
\end{align*}
where $L(t)=e^{itH_0} H e^{-itH_0}-H_0$. We recall, for any reasonable symbol 
$a(x,\x)$, we have 
\[
e^{itH_0} a^w(x,D_x) e^{-itH_0} = (a\circ \exp(tH_{p_0}))^w(x, D_x),
\]
without the remainder terms, since $p_0(x,\x)$ is a quadratic form in $(x,\x)$, 
and we employ the Weyl quantization. Thus we have 
\[
L(t) =(p\circ \exp(tH_{p_0}))^w(x,D_x) -p_0(x,D_x) = \ell^w(t;x,D_x), 
\]
i.e., $\ell(t;x,\x)$ is the Weyl-symbol of $L(t)$. This is in fact expected, since 
$e^{itH_0}e^{-itH}$ is the quantization of $S_t$. 

Let $\Omega$ be a small neighborhood of $(x_0,\x_0)$ as in the last section, 
and let $f\in C_0^\infty(\Omega)$ be such that $f(x_0,\x_0)>0$, 
and $f(x,\x)\geq 0$ on $\re^{2n}$. We then set
\[
f_h(x,\x)= f(x,h\x), \quad h=\l^{-1}, 
\]
where $h>0$ is our semiclassical parameter. We consider the behavior of 
\[
G(t) = e^{itH_0} e^{-itH} f^w(x,h D_x) e^{itH}e^{-itH_0}
\]
as $h\to 0$. The operator valued function $G(t)$ satisfies the Heisenberg equation:
\begin{equation}
\frac{d}{dt} G(t) = -i [L(t), G(t)], \quad G(0)= f^w(x,h D_x).
\end{equation}
The corresponding canonical equation of the classical mechanics is 
\[
\frac{\pa \g_0}{\pa t}(t;x,\x) = -\{\ell,\g_0\}(t;x,\x), 
\quad \g(0;x,\x)= f(x,h \x)
\]
and the solution is given by 
\[
\g_0(t;x,\x) = (f_h\circ S_t^{-1})(x,\x)
\]
since $S_t$ is the Hamilton flow generated by $\ell(t;x,\x)$. 
Now we note 
\[
f_h(x,\x) = f(x,\x/\l)= (f\circ \mathcal{J}_\l^{-1})(x,\x).
\]
Hence, recalling (2.6), we learn
\begin{align*}
f_h\circ S_t^{-1} &= f\circ \mathcal{J}_\l^{-1}\circ S_t^{-1} 
= f\circ (S_t\circ \mathcal{J}_\l)^{-1} \\
&= f\circ (\mathcal{J}_\l \circ S^\l_{\l t})^{-1} 
= f\circ (S_{\l t}^\l)^{-1}\circ \mathcal{J}_\l^{-1}.
\end{align*}
In other words, we have 
\[
\g_0(t;x,\x) = (f_h \circ S_t^{-1})(x,\x) = (f\circ (S_{\l t}^\l)^{-1})(x,h\x). 
\]
We expect 
\[
G(t)\sim \g_0(t;x,D_x) = (f\circ (S_{\l t}^\l)^{-1})(x,h D_x)
\]
for small $h>0$, and we construct the asymptotic solution to the 
Heisenberg equation (3.1) with the principal symbol $\g_0(t;x,\x)$. 

\begin{lem}
Let $-\pi<t_0<0$, and set $I=[t_0,0]$. There exists $\g(t;x,\x)\in C_0^\infty(\re^{2n})$
for $t\in I$ such that 
\begin{enumerate}
\renewcommand{\theenumi}{\roman{enumi}}
\renewcommand{\labelenumi}{(\theenumi)} 
\item $\g(0;x,\x) = f(x,h\x)$. 
\item $\g(t;x,\x)$ is supported in $S_t\circ \mathcal{J}_\l(\Omega) = \mathcal{J}_\l
\circ S_{\l t}^\l (\Omega)$. 
\item For any $\a,\b\in\ze_+^n$, there is $C_{\a\b}>0$ such that 
\[
\bigabs{\pa_x^\a \pa_\x^\b \g(t;x,\x)}\leq C_{\a\b}h^{|\b|}, 
\quad t\in I, x,\x\in\re^n.
\]
\item The principal symbol of $\g$ is given by $\g_0$, i.e., 
for any $\a,\b\in\ze_+^n$, there is $C_{\a\b}>0$ such that 
\[
\bigabs{\pa_x^\a \pa_\x^\b (\g(t;x,\x)-\g_0(t;x,\x))}\leq C_{\a\b}h^{1+|\b|}, 
\quad t\in I, x,\x\in\re^n.
\]
\item If we set $G(t)=\g^w(t;x,D_x)$, then 
\[
\biggnorm{\frac{d}{dt} G(t) +i [L(t),G(t)]}_{\mathcal{L}(\mathcal{H})} 
=O(h^\infty)
\]
as $h\to 0$, uniformly in $t\in I$.  
\end{enumerate}
\end{lem}

\begin{proof}
Given the classical mechanical construction above, the construction of the 
asymptotic solution is quite similar to (or slightly simpler than) 
the proof of Lemma~4 of \cite{Na1}. 
We note $\ell(t;x,\x)\in S(\jap{\x}^2, dx^2+d\x^2/\jap{\x}^2)$ locally in 
$x$, and $\pi_1(\supp(\g_0(t;\cdot,\cdot)))$ is contained in a compact set by virtue of the 
asymptotic property: $S_{\l t}^\l \sim S_+$ as $\l\to\infty$. We omit the detail. 
\end{proof}

Now the proof of Theorem~1.1 is almost the same as the proof of Theorem~1 of 
\cite{Na1}, and we simply refer the reader to the paper. 

Finally, we show Theorem~1.3 follows from Theorem~1.1. 

\begin{proof}[Proof of Theorem 1.3]
We prove the part (i) only. 
We note 
\begin{align}
\WF(e^{-i(\pi/2)H_0}u)&= \WF(\hat u), \\
\WF(e^{i(\pi/2)H_0} u)&= \WF(\check u) =\Gamma(\WF(\hat u)).
\end{align}
In fact, $e^{-i(\pi/2)H_0}$ is the Fourier transform. 
(See also Remark~1.2.) 

We set $(x',\x')= S_+^{-1}\circ \Gamma\circ S_-(x_0,\x_0)$ so that 
\begin{equation}
(x_-,\x_-) =\Gamma(x_+',\x_+'), \quad \text{where } (x_+',\x_+')=S_+(x',\x'). 
\end{equation}
By Theorem~1.1 (i) with $t_0=\pi/2$, and $u(\pi/2)$ as the initial condition, we have 
\begin{align*}
(x_0,\x_0)\in \WF(u(\pi)) 
\quad &\Longleftrightarrow \quad
(x_-,\x_-)\in \WF(e^{-i(\pi/2)H_0} u(\pi/2)) \\
\quad &\Longleftrightarrow \quad
(x_-,\x_-)\in \WF(\hat u(\pi/2)).
\end{align*}
We have used (3.2) in the second step. 
On the other hand, by Theorem~1.1 (ii) with $t_0=-\pi/2$, and $u(\pi/2)$ 
as the initial condition, and using (3.3), we also have 
\begin{align*}
(x',\x')\in \WF(u_0) 
\quad \Longleftrightarrow \quad
(x'_+,\x'_+)\in \ &\WF(e^{i(\pi/2)H_0} u(\pi/2)) \\
&= \WF(\check u(\pi/2)) =\Gamma(\WF(\hat u(\pi/2))). \end{align*}
By (3.4), this implies the claim of Theorem~1.3 (i). 
The part (ii) is proved similarly. 
\end{proof}


\section{Inhomogeneous harmonic oscillators}

Here we consider the case when the harmonic potential is inhomogeneous, i.e., 
\[
H_0= -\frac12 \triangle +\frac12 \sum_{i,j=1}^n b_{ij} x_ix_j,
\]
with a positive symmetric matrix  $(b_{ij})$, and
\[
H= -\frac12 \sum_{i,j=1}^n \pa_{x_i} a_{ij}(x) \pa_{x_j} 
+\frac12 \sum_{i,j=1}^n b_{ij} x_ix_j +V(x).
\]
We assume $(a_{jk}(x))$ and $V(x)$ satisfy Assumption~A. 
By an orthogonal transform, we can diagonalize  the harmonic potential, and 
hence we may assume $\sum b_{ij} x_i x_j =\sum_{j=1}^n \n_j^2 x_j^2$, where 
$\n_j^2>0$, $j=1,\dots, n$, are eigenvalues of $(b_{ij})$. The behavior of the 
inhomogeneous harmonic oscillator depends on the number theoretical properties 
of $(\n_j)_{j=1}^n$. If there exist no $t_0>0$ such that 
\begin{equation}\label{res}
t_0 \n_j \in\pi\ze, \quad j=1,\dots, n,
\end{equation}
then it is well-known that the recurrence of the evolution operator does not 
occur, i.e., there are no $t_0\neq 0$ such that $e^{-itH_0}=I$. 
In this case we have the following result:

\begin{thm}
Suppose $(x_0,\x_0)$ is backward nontrapping, and suppose that 
there are no $t_0>0$ such that (\ref{res}) hold. Then for any $t>0$, 
\[
(x_0,\x_0)\in \WF(e^{-itH} u_0) \quad \Longleftrightarrow \quad 
(x_-,\x_-)\in \WF(e^{-itH_0} u_0).
\]
\end{thm}

Obviously, an analogous result holds for $t<0$, but we omit it here. 

If there exists $t_0>0$ such that (\ref{res}) holds, then we have the following result:

\begin{thm}
Let $t_0>0$ be the smallest positive number satisfying (\ref{res}), 
and let $m_j=t_0\n_j/\pi\in \ze$. We set 
\[
\tilde \Gamma (x_1,\dots,x_n,\x_1,\dots,\x_n) =
(\s_1x_1,\dots, \s_n x_n,\s_1\x_1,\dots,\s_n\x_n)
\]
for $(x,\x)\in\re^{2n}$, 
where $\s_j=1$ if $m_j$ is even, and $\s_j=-1$ if $m_j$ is odd. 
Suppose $(x_0,\x_0)$ is backward nontrapping. Then for $0<t<t_0$, 
\[
(x_0,\x_0)\in \WF(e^{-itH} u_0) \quad \Longleftrightarrow \quad 
(x_-,\x_-)\in \WF(e^{-itH_0} u_0), 
\]
and 
\[
(x_0,\x_0)\in \WF(e^{-it_0H} u_0) \quad \Longleftrightarrow \quad 
S_+^{-1}\circ \tilde \Gamma \circ S_-(x_0,\x_0)\in \WF(u_0).
\]
\end{thm}

The proofs of these theorems are similar to Theorems~1.1 and 1.3,  and 
we omit the detail. We only note the fact that 
\[
\exp\biggbrac{-it_0\biggpare{-\frac12\frac{d^2}{dx^2} +\n_j^2\frac{x^2}{2}}} u(x) 
=(\mathcal{F}^{2m_j} u)(x) = u(\s_j x)
\]
for $u\in L^2(\re)$, $j=1,\dots, n$.


\end{document}